\documentclass{amsart}
\usepackage{amsfonts}
\usepackage{amsmath,amssymb}
\usepackage{amsthm}
\usepackage{amscd}
\usepackage{enumitem}
\usepackage{graphics}
\usepackage{graphicx}

\theoremstyle{remark}{
\newtheorem{Def}{{\rm Definition}}

\newtheorem{Rem}{{\rm Remark}}
\newtheorem{Prob}{{\rm Problem}}

}
\theoremstyle{plain}
{

\newtheorem{Prop}{Proposition}
\newtheorem{Thm}{Theorem}

}

\begin{document}
\title[Compactifying real analytic functions and resulting Reeb spaces]{Compactifying real analytic functions and resulting Reeb spaces}
\author{Naoki kitazawa}
\keywords{Smooth, real analytic, or real algebraic (real polynomial) functions and maps. Compactifications of spaces and (differentiable) smooth, real algebraic, real analytic, or continuous maps. Reeb spaces. Cell complexes. 1-dimensional cell complexes whose edges are oriented. Peano continua. (Di)graphs. Reeb (di)graphs. \\
\indent {\it \textup{2020} Mathematics Subject Classification}: Primary~26C05, 26E05, 54C30, 54F15, 57R45, 58C05.}

\address{Osaka Central Advanced Mathematical Institute (OCAMI) \\
3-3-138 Sugimoto, Sumiyoshi-ku Osaka 558-8585
TEL: +81-6-6605-3103
}
\email{naokikitazawa.formath@gmail.com}
\urladdr{https://naokikitazawa.github.io/NaokiKitazawa.html}
\maketitle
\begin{abstract}
We formulate compactifications of continuous maps naturally. We consider real analytic functions mainly. We are interested in topological properties and combinatorial ones of explicit resulting maps. For understanding them, we use their {\it Reeb spaces}, being quotient spaces of the spaces of the domains of the functions and defined by the equivalence relation identifying two points in same components of their level sets. They are known to be $0$- or $1$-dimensional (metrizable) cell-complexes, in our situations or more general certain tame cases. 

Reeb spaces have been important in understanding topological properties and combinatorial ones of functions and spaces roughly, since the last century.

These compactifications have been explicitly studied by the author previously and recently. We have obtained real algebraic functions whose Reeb spaces are not so complicated and which seem to be of most natural and simplest. We present new discussions and examples.

\end{abstract}
%【REVISE】 combinatoric ～ is → combinatorial object. It is .
%【REVISE】  such that a point is a vertex if and only if the corresponding connected component of the level set contains some singular points → whose vertex set is the set of all points containing some singular points in the corresponding connected component of the level set .
%【REVISE】 We delete "extending the result before".
\section{Introduction.}
\label{sec:1}
\subsection{Reeb spaces?}
The {\it Reeb spaces} $R_c$ of continuous functions $c:X \rightarrow Y$ where $Y$ is a real number field $\mathbb{R}$ (the case of real-valued functions) or the unit circle $S^1:=\{(x_1,x_2) \in {\mathbb{R}}^2 \mid  {x_1}^2+{x_2}^2=1\} \subset {\mathbb{R}}^2$  (the case of circle-valued functions) on topological spaces $X$ are fundamental and strong tools in understanding $X$ and topological properties and combinatorial properties of the functions. \cite{reeb} is one of related pioneering papers. 

We give its rigorous definition. Let $c:X \rightarrow Y$ be a continuous map.
We have the equivalence relation ${\sim}_c$ on $X$ and the quotient space $R_c:=X/{\sim c}$: for $x_1,x_2 \in X$, $x_1 {\sim}_c x_2$ if and only if they are in a same connected component of a {\it preimage} $c^{-1}(y)$. A {\it level set} of the function $c:X \rightarrow Y:=\mathbb{R}, S^1$ is a preimage $c^{-1}(y)$. A {\it contour} of $c$ is a connected component of a preimage of of $c$. We have the quotient map $q_c:X \rightarrow R_c$ and the unique continuous map $\bar{c}:R_c \rightarrow Y$ with $c=\bar{c} \circ q_c$.

They have been shown to be graphs for certain nice real-valued functions on (compact) manifolds such as Morse(-Bott) functions and naturally generalized ones, in \cite{izar, martinezalfaromezasarmientooliveira}, and shown in \cite{saeki1, saeki2}, for real-valued functions, more generally. We can apply this for circle-valued cases, similarly, where we do not need to know the arguments.

We explain singularity of smooth maps. A {\it singular} point of a differentiable map $c:X \rightarrow Y$ is a point of the manifold $X$ where the rank of the differential is smaller than both the dimensions of $X$ and $Y$. Let $S(c)$ denote the set of all singular points of $c$ and it is the {\it singular set} of $c$. We use "{\it critical}" instead of "singular" in the case where the manifold of the target is $1$-dimensional.

For the graphs above, a point $v$ is a vertex if ${q_c}^{-1}(v)$ contains some critical point of $c$. A contour of $c$ is said to be {\it critical}. A contour of $c$ which is not critical is {\it regular}. $R_c$ is the {\it Reeb graph} of $c$ and by the rule that each edge $e$ incident to $v_1$ and $v_2$ is oriented as one departing from $v_1$ and entering $v_2$ induced from $\mathbb{R}$, oriented naturally by the canonical orders on $\mathbb{R}$, or $S^1$, oriented in the counter clockwise way, it is a digraph (the {\it Reeb digraph} of $c$)

We explain general theory such as \cite{gelbukh1, gelbukh2, gelbukh3, gelbukh4, saeki1, saeki2}. For smooth real-valued functions on closed and connected manifolds and more generally, continuous real-valued functions on compact, connected, locally connected and spaces ({\it Peano continua}) which are metrizable, the Reeb spaces are (at most $1$-dimensional and metrizable) Peano continua. For certain general (real-valued) continuous functions on topological spaces, it is studied whether the Reeb spaces are Hausdorff. 
\subsection{Reconstructing nice smooth, real algebraic, or real analytic functions with given Reeb graphs.}
Related to this, interest of the author lies in reconstruction of nice smooth functions with given Reeb graphs. This has been first launched by Sharko in \cite{sharko} and followed by Masumoto, Michalak, Saeki, and so on, in \cite{masumotosaeki, michalak} for example. These are essentially on reconstruction of nice smooth functions on closed surfaces which are locally Morse or represented as certain elementary polynomials around critical points of them. The author has also contributed to it with their higher dimensional cases, respecting topologies of level sets in addition. See papers such as \cite{kitazawa1, kitazawa4}. For related studies on non-compact (proper) cases, see \cite{kitazawa2} for example. \cite{kitazawa5} is also a related preprint of the author.

Later, the author has been interested in explicit construction of functions with stronger regularity, in other words, {\it real algebraic}, or {\it real analytic} functions. In the closed manifold cases, the author has systematically constructed and shown realizations of given Reeb (di)graphs of certain classes by constructing real algebraic functions. The author has generalized the so-called natural height function of the $m$-dimensional unit sphere $S^m:=\{x:=(x_1 \cdots x_{m+1}) \in {\mathbb{R}}^{m+1} \mid {\Sigma}_{j=1}^{m+1} {x_j}^2=1 \}$ in the ($m+1$)-dimensional Euclidean space ${\mathbb{R}}^{m+1}$ with $m \geq 1$. It is the restriction of the projection to the first component by the projection ${\pi}_{m+1,1}(x):=x_1$. The canonical projection ${\pi}_{k,k_1}:{\mathbb{R}^{k}} \rightarrow {\mathbb{R}}^{k_1}$ is defined by ${\pi}_{k,k_1}(x):=x_1$ ($x:=(x_1.x_2) \in {\mathbb{R}}^{k_1} \times {\mathbb{R}}^{k_2}={\mathbb{R}}^{k}$). For related studies, see the pioneering paper \cite{kitazawa3} and see also \cite{kitazawa6, kitazawa7}.

As related another study, in \cite{basucoxpercival}, Reeb spaces of so-called {\it proper} {\it definable} maps of a certain class are shown to be realized as so-called {\it proper} {\it definable quotients}, where the author is not specialized in related theory. 

We use $c {\mid}_Z$ for the restriction of a map $c:X \rightarrow Y$ to a subset $Z\subset X$. Here, an $m$-dimensional {\it real analytic} ({\it algebraic}) manifold $M \subset {\mathbb{R}}^{m+k}$ means a union of connected components of the zero set of a {\it real analytic} (resp. {\it polynomial}) map $e_{m,k}:{\mathbb{R}}^{m+k} \rightarrow {\mathbb{R}}^k$ such that the restriction $e_{m,k} {\mid}_{M}$ has no singular point. A {\it real analytic} (resp. {\it algebraic}) map means a map represented as ${\pi}_{m+k,k^{\prime}} {\mid}_{M}$ with $m+k>k^{\prime} \geq 1$.

We discuss Problem \ref{prob:1}, which are also questioned and solved with explicit cases in \cite{kitazawa8} first, and in \cite{kitazawa9, kitazawa10, kitazawa11, kitazawa12} for example.
\begin{Prob}
\label{prob:1}
How should we investigate realization of Reeb spaces which are not homeomorphic to finite graph by nice smooth functions?
\end{Prob}

\begin{Prob}
\label{prob:2}
Can we explain nice examples for Problem \ref{prob:1}?
\end{Prob}

Related to this, see also This is also presented in \cite[Problem 1]{gelbukh5}.

This has been first answered affirmatively in \cite{kitazawa8} (\cite[Theorems 1 and 2]{kitazawa8}). This study is followed by the author himself exclusively. See \cite{kitazawa9, kitazawa10, kitazawa11} and a recent preprint \cite{kitazawa12}. Our main result consists of new result related to these Problems, with new explicit arguments.

In the next section, we review main ingredients of \cite{kitazawa12}, a method for representing Reeb spaces which are not homeomorphic to any finite graph, by simpler oriented cell-complexes of dimension $1$, compactly. We introduce our main result (Theorems \ref{thm:1} and \ref{thm:2}).
In the third section, we review arguments important in our present study. In the fourth section, we prove our main result and have their refined versions as Theorems \ref{thm:3} and \ref{thm:4}.

\section{Our main result.}
For a topological space $X$ and its subspace $Y$, ${\overline{Y}}^X$ is for the closure of $Y$ in $X$.

We use $D^k:=\{x:=(x_1 \cdots x_{n}) \in {\mathbb{R}}^{k} \mid {\Sigma}_{j=1}^{k} {x_j}^2 \leq 1 \}$ in the $k$-dimensional Euclidean space ${\mathbb{R}}^{k}$ with $k \geq 1$.
We use  $S^0:=\{-1,1\} \subset \mathbb{R}$. Let $D^0$ mean a one-point set and let $S^{-1}$ mean the emptyset $\emptyset$.

A {\it cell complex} $(X,\{\{e_{j_1,j_2}\}_{j_2 \in J_{j_1}}\}_{j_1=0}^d)$ of dimension $d$ means a pair of a Hausdorff space $X$ and its underlying set ${\sqcup}_{j_1=0}^d ({\sqcup}_{j_2 \in J_{j_1}} e_{j_1,j_2})$ and with the topology as follows.
\begin{itemize}
\item $e_{j_1,j_2}$ is a subspace of $X$ homeomorphic to $D^{j_1}-S^{j_1} \subset {\mathbb{R}}^{j_1}$ and a {\it $j_1$-cell} of the cell complex. 
\item The closure of $e_{j_1,j_2}$ in $X$ is a union of $e_{j_1,j_2}$ and some of $e_{{j_1}^{\prime},{j_2}^{\prime}}$ with ${j_1}^{\prime}<j_1$.
\end{itemize}
As a specific case, a cell complex with the following is a {\it CW complex}.
\begin{itemize}
\item The closure of $e_{j_1,j_2}$ in $X$ is a union of $e_{j_1,j_2}$ and finitely many sets $e_{{j_1}^{\prime},{j_2}^{\prime}}$ with ${j_1}^{\prime}<j_1$.
\item $O \subset X$ is open in $X$ if and only if for each $e_{j_1,j_2}$, $O \bigcap e_{j_1,j_2}$ is open in $e_{j_1,j_2}$.
\end{itemize}
The following are well-known of important facts on general topology, where we abuse the notation.
\begin{Prop}
For a topological space $X$ with the structure of a certain cell complex by considering suitable ${\sqcup}_{j_1=0}^d ({\sqcup}_{j_2 \in J_{j_1}} e_{j_1,j_2})$, $d$ is a topological invariant for $X$.
\begin{enumerate}
\item For a topological space $X$ which has the structure of a CW complex by considering suitable ${\sqcup}_{j_1=0}^d ({\sqcup}_{j_2 \in J_{j_1}} e_{j_1,j_2})$, $d$ is a topological invariant for $X$.
\item For a topological space $X$ which has the structure of a cell complex by considering suitable ${\sqcup}_{j_1=0}^d ({\sqcup}_{j_2 \in J_{j_1}} e_{j_1,j_2})$ and which is metrizable, $d$ is a topological invariant for $X$.
\end{enumerate}
\end{Prop}

We consider $d$ as the {\it dimension} of such a space $X$ and $X$ (the cell complex) can be referred to be a {\it $d$-dimensional} cell complex.

In \cite{kitazawa12}, we consider a connected Hausdorff space $G$ and a triplet $(G,S_G,c_G:G \rightarrow Y)$ ($Y:=\mathbb{R}, S^1$) with the following. We call this triplet a {\it pre-digraph}. The case $Y:=S^1$ is explicitly discussed in the present paper, where related exposition is in \cite{kitazawa12}, shortly.
\begin{itemize}
\item $S_G$ is a closed subset of $G$ and $G-S_G$ is a disjoint union ${\sqcup}_{j \in J_G} e_{G,j}$ of copies $e_{G,j}$ of  $\mathbb{R}$. By regarding $e_{G,j}$ as $1$-cells and points of $S_G$ as $0$-cells, we have a $1$-dimensional cell complex. The restriction $C_G {\mid}_{e_{G,j}}$ is injective.

\item ${\rm NF}(G,S_G)$ is the set of all points of which we cannot have connected neighborhoods being homeomorphic to some finite and connected graphs and being disjoint from $S_G$. We call such a point a {\it non-finite} point of $(G,S_G)$.
\end{itemize}
We consider a set of all connected components of $G-S_G$. For two connected components $C_{G-{\rm NF}(G,S_{G}),c_{G},1}$ and $C_{G-{\rm NF}(G,S_{G}),c_{G},2}$ of $G-{\rm NF}(G,S_{G})$, if and only if at least one of the following holds, then $C_{G-{\rm NF}(G,S_{G}),c_{G},1} {\sim}_{G-{\rm NF}(G,S_{G}),c_{G}} C_{G-{\rm NF}(G,S_{G}),c_{G},2}$. 
\begin{itemize}
\item For some point $s \in {\rm NF}(G,S_{G})$, there exist non-empty families $\{e_{G,j_1}\}_{j_1 \in J_{G,C_{G-{\rm NF}(G,S_{G}),c_{G},1}} \subset J_G} \subset C_{G-{\rm NF}(G,S_{G}),c_{G},1}$ and $\{e_{G,j_2}\}_{j_2 \in J_{G,C_{G-{\rm NF}(G,S_{G}),c_{G},1}} \subset J_G} \subset C_{G-{\rm NF}(G,S_{G}),c_{G},2}$ such that the restrictions of $c_G$ to the closures of their disjoint unions in $G$ have the maximum value at $s$: $s$ is said to be an {\it ascending} point for these connected components. 
\item  For some point $s \in {\rm NF}(G,S_{G})$, there exist non-empty families $\{e_{G,j_1}\}_{j_1 \in J_{G,C_{G-{\rm NF}(G,S_{G}),c_{G},1}} \subset J_G} \subset C_{G-{\rm NF}(G,S_{G}),c_{G},1}$ and $\{e_{G,j_2}\}_{j_2 \in J_{G,C_{G-{\rm NF}(G,S_{G}),c_{G},1}} \subset J_G} \subset C_{G-{\rm NF}(G,S_{G}),c_{G},2}$ such that the restrictions of $c_G$ to the closures of their disjoint unions in $G$ have the minimum value at $s$: $s$ is said to be a {\it descending} point for these connected components .
\end{itemize}
From this, we can uniquely define the equivalence relation ${\sim}_{G-{\rm NF}(G,S_{G}),c_{G}}$ on the sets of all connected components of $G-S_G$. 
More precisely, we consider the equivalence relation generated by the defined relation. We can have a $0$- or $1$-dimensional cell-complex the set of all of whose $0$-cells consists of all these equivalence classes $[C_{G-{\rm NF}(G,S_{G}),c_{G}}]$. 
We can define {\it ascending} points and {\it descending} points for the equivalence classes naturally. 
Furthermore, the set of all its $1$-cells consists of all elements of ${\rm NF}(G,S_G)$ and each of the element departs from its $0$-cell $v_{[C_{G-{\rm NF}(G,S_{G}),c_{G}}]}$ for the equivalence class $[C_{G-{\rm NF}(G,S_{G}),c_{G}}]$ if and only if it is for an ascending point for $[C_{G-{\rm NF}(G,S_{G}),c_{G}}]$ and each of  its $1$-cells enters $v_{C_{G-{\rm NF}(G,S_{G}),c_{G}}}$ if and only if it is for a descending point for $[C_{G-{\rm NF}(G,S_{G}),c_{G}}]$. 
\begin{Def}
\label{def:1}
The resulting cell complex is the {\it graph diagram for NF of $(G,S_G,c_G:G \rightarrow \mathbb{R})$} and denoted by ${\rm GDNF}(G,S_G,c_G)$.
As a specific class, if for each $[C_{G-{\rm NF}(G,S_{G}),c_{G}}]$, exactly one connected component of $G-S_G$ exists, then ${\rm GDNF}(G,S_G,c_G)$ is said to be {\it normal}.
\end{Def}
	{\it Reeb digraphs} are regarded to be specific cases of pre-digraphs. We redefine the {\it Reeb digraph} of $c:X \rightarrow Y$.
\begin{Def}
 \label{def:2}
The {\it Reeb digraph} of $c:X \rightarrow Y$ can be defined by $(R_c,S(c),\bar{c})$ if it is a pre-digraph.
\end{Def}

We can naturally define {\it isomorphisms} between pre-digraphs by homeomorphisms mapping important objects in the canonical way. By the existence of such isomorphisms we can define two isomorphic pre-digraphs. 

We have reviewed a method for representing $0$- or $1$-dimensional cell complexes which may not be homeomorphic to any graph, respecting \cite{kitazawa12}. This is essentially same as the original exposition, where some way of exposition is a bit different from the original one.
\begin{Def}
\label{def:3}
Let $\mathcal{C}$ denote the category such as the differentiable (smooth), real algebraic, real analytic, or topological one.
For a map $c:X \rightarrow Y$ in $\mathcal{C}$, let ${\phi}_X: X \rightarrow X_0$ and ${\phi}_Y:Y \rightarrow Y_0$ denote embeddings in $\mathcal{C}$ such that the image ${\phi}_X (X)$ is dense in a compact and connected space $X_0$ and that the image ${\phi}_Y (Y)$ is dense in a connected space $Y_0$. 
We also consider another category $\mathcal{C^{\prime}}$ obtained by a forgetful functor from $\mathcal{C}$.
A map $c_0:X_0 \rightarrow Y_0$ in the category ${\mathcal{C}}^{\prime}$ satisfying the relation $c_0 {\mid}_{\phi(X)}={\phi}_Y \circ c  \circ {{\phi}_X}^{-1} {\mid}_{\phi(X)}$ is a {\it compactification of} $c$ {\it in} $\mathcal{C^{\prime}}$. 
\end{Def}
This is a kind of fundamental arguments on topological spaces and continuous maps between them. This is also important in rational maps in algebraic (analytic) geometry (, where the author is not specialized in real algebraic geometry and studying algebraic geometry from fundamental notions and idea). Note that a map of this type explicitly appears in \cite{kitazawa12}. This is also closely related to Theorem \ref{thm:2}.

\begin{Def}
\label{def:4}
	For the Reeb digraph $(R_c,S(c),\bar{c})$ of $c:X \rightarrow Y$, if we can define a compactification $c_0:X_0 \rightarrow Y_0$ ($Y_0:=\mathbb{R}, S^1$) of $c$ in the topology category and have a pre-digraph $(R_{c_0},q_{c_0}({\phi}_X(S(c)) \bigcup (X_0-{\phi}_X(X)),\bar{c_0})$, then $(R_{c_0},q_{c_0}({\phi}_X(S(c)) \bigcup (X_0-{\phi}_X(X)),\bar{c_0})$ is said to be a {\it Reeb digraph compactification} or a {\it Reeb-D-C} of $c$. Let it be denoted by ${\rm ReebDC}(c,c_0,{\phi}_X,{\phi}_Y)$.
\end{Def}
Specific cases of this class are studied in \cite{kitazawa12}. %This is presented in the Subsection \ref{subsec:3.4}.
We introduce one of our main result.
\begin{Thm}
\label{thm:1}
For an integer $m>1$ and non-negative integers $n_{\rm d}$, $n_{\rm e}$ and $n_{\rm c}$, there exists an $m$-dimensional real analytic manifold $X_{m,n_{\rm d,e,c}} \subset {\mathbb{R}}^{m+1}$ and the following hold.
\begin{enumerate}
\item We have a Reeb-D-C of ${\pi}_{m+1,1} {\mid}_{X_{m,n_{\rm d,e,c}}}$.
\item Furthermore, we can have the Reeb-D-C of ${\pi}_{m+1,1} {\mid}_{X_{m,n_{\rm d,e,c}}}$ in such a way that
its graph diagram for NF is normal and with exactly one vertex and the following edges incident to it.
\begin{enumerate}
\item Exactly $n_{\rm d}$ edges departing from this whose closures in the whole {\rm (}$1$-dimensional{\rm )} space are homeomorphic to $\{t \mid t>0\}$.
\item Exactly $n_{\rm e}$ edges entering this whose closures in the whole {\rm (}$1$-dimensional{\rm )} space are homeomorphic to $\{t \mid t>0\}$.
\item Exactly $n_{\rm c}$ edges whose closures in the whole {\rm (}$1$-dimensional{\rm )} space are homeomorphic to $S^1$.
 \end{enumerate}
\end{enumerate}
\end{Thm}
We can say that Theorem \ref{thm:1} is partially shown in \cite{kitazawa12} for specific cases presented shortly above and reviewed shortly in the Subsection \ref{subsec:3.4} and that this is shown in a different way and different story.

We also prove the following as another main result.
\begin{Thm}
\label{thm:2}
For an integer $m>1$, there exists a pair of an $m$-dimensional real analytic manifold $X_{m,1} \subset {\mathbb{R}}^{m+2}$ diffeomorphic to ${\mathbb{R}}^m$ and an $m$-dimensional smooth compact submanifold $X_{m,2} \subset {\mathbb{R}}^{m+2}$ diffeomorphic to $S^m$ and the following hold.
\begin{enumerate}
\item \label{thm:2.1} $X_{m,1}$ is diffeomorphic to ${\mathbb{R}}^m$. $X_{m,2}$ is diffeomorphic to $S^m$. 
\item \label{thm:2.2} The Reeb digraphs of the functions ${\pi}_{m+2,1} {\mid}_{X_{m,1}}$ and ${\pi}_{m+2,1} {\mid}_{X_{m,2}}$ are defined.
\item \label{thm:2.3} The function ${\pi}_{m+2,1} {\mid}_{X_{m,2}}$ is a compactification of  the function ${\pi}_{m+1,1} {\mid}_{X_{m,1}}$ in the smooth category.
\item \label{thm:2.4} The Reeb digraph of ${\pi}_{m+2,1} {\mid}_{X_{m,2}}$ is ${\rm ReebDC}({\pi}_{m+1,1} {\mid}_{X_{m,1}},{\pi}_{m+2,1} {\mid}_{X_{m,2}},{\phi}_X,{\phi}_Y)$ for some pair $({\phi}_X:X_{m,1} \rightarrow X_{m,2},{\phi}_Y:\mathbb{R} \rightarrow \mathbb{R})$ of an embedding ${\phi}_X$ and the identity map ${\phi}_Y:\mathbb{R} \rightarrow \mathbb{R}$. Here, in addition, ${\rm ReebDC}({\pi}_{m+1,1} {\mid}_{X_{m,1}},{\pi}_{m+2,1} {\mid}_{X_{m,2}},{\phi}_X,{\phi}_Y)$ and the Reeb digraph of ${\pi}_{m+1,1} {\mid}_{X_{m,1}}$ are not isomorphic.  
\end{enumerate}
\end{Thm}
\section{Fundamental and important methods in our proof of our main result.}
\label{sec:3}
\subsection{Affine transformations.}
\label{subsec:3.1}
Here, we review affine transformations on the real affine space ${\mathbb{R}}^k$. An {\it affine transformation} there means a map obtained by the composition of finitely many diffeomorphisms each of which is a linear transformation or a parallel transformation there. For linear transformations, rotations around the origin $0 \in {\mathbb{R}}^k$ is important, for example.

As a simplest example, an {\it ellipsoid} means a subset of ${\mathbb{R}}^2$ obtained as the image of the restriction to $D^2 \subset {\mathbb{R}}^2$ of an affine transformation on ${\mathbb{R}}^2$ and its boundary is a real algebraic manifold of dimension $1$ and diffeomorphic to $S^1$. 

Related to this, we define a {\it hyperbola} and its component.
\subsection{A hyperbola and its component.}
\label{subsec:3.2}
A {\it component} of a {\it hyperbola} is a $1$-dimensional real algebraic manifold in ${\mathbb{R}}^2$, defined as the graph $\{(x,\sqrt{x^2+1}) \mid x \in \mathbb{R}\} \subset {\mathbb{R}}^2$.
Related to the asymptotic behavior, the lines $\{(x_1,x_2) \mid x_1 \pm x_2=0\}$ are {\it asymptotes} of this. The 1st derivative of the function $p_{\rm h}(x):\mathbb{R} \rightarrow \mathbb{R}$, defined by $p_{\rm h}(x):=\sqrt{x^2+1}$, is ${p_{\rm h}}^{\prime}(x)=\frac{x}{\sqrt{x^2+1}}$.
A real algebraic manifold in ${\mathbb{R}}^2$ obtained as the image of the restriction there of an affine transformation on ${\mathbb{R}}^2$ can be also defined to be a {\it component of a hyperbola} and its {\it asymptote} is also defined similarly. It is regarded to be a real algebraic manifold in ${\mathbb{R}}^2$ obtained as the image of the restriction to $\{(x_1,x_2) \in {\mathbb{R}}^2 \mid {x_1}^2-{x_2}^2=1, x_2>0\}$ of an affine transformation on ${\mathbb{R}}^2$.
\subsection{The asymptotic behavior of real valued functions $q_{p,s,i}:\mathbb{R} \rightarrow \mathbb{R}$ and their 1st derivatives $q_{p,s,i}^{\prime}$.}
\label{subsec:3.3}
We consider a function of the form $q_{p,s,1}(x):=\frac{{p_1(x)}^{s_1}}{{p_2(x)}^{s_2}}+\frac{\sin{(e^{-x})}{p_1(x)}^{s_1}}{2{p_2(x)}^{s_2}}$, $q_{p,s,2}(x):=\frac{{p_1(x)}^{s_1}}{{p_2(x)}^{s_2}}+\frac{\sin{(e^{x})}{p_1(x)}^{s_1}}{2{p_2(x)}^{s_2}}$ and $q_{p,s,3}(x):=\frac{{p_1(x)}^{s_1}}{{p_2(x)}^{s_2}}+\frac{\sin{(e^{x^2})}{p_1(x)}^{s_1}}{2{p_2(x)}^{s_2}}$ with the following.
\begin{itemize}
\item  $p_1(x)$ and $p_2(x)$ are real polynomials being always positive for $x \in \mathbb{R}$.
\item  $s_1, s_2>0$, and $\frac{{p_1(x)}^{s_1}}{{p_2(x)}^{s_2}}$ converges to $0$ as $x$ diverges to $\pm \infty$.
\end{itemize}

We can also see the following.
\begin{itemize}
\item The value $q_{p,s,i}(x)$ is positive for any $x \in \mathbb{R}$.
\item $q_{p,s,i}(x)$ converges to $0$ as $x$ diverges to $\pm \infty$.
\item The following are on asymptotic behaviors of 1st derivatives ${q_{p,s,i}}^{\prime}$. In \cite{kitazawa11}, the author has discussed this explicitly first for our studies on topological properties and combinatorial ones of Reeb spaces. We also discuss this in \cite{kitazawa12}. We consider the orders of the divergence at $\pm \infty$. We can check in a self-contained way and we refer to the original preprints for the understanding. 
We consider the {\it limit superiors} ${\limsup}_{x \to \pm \infty} c(x)$ and the {\it limit inferiors} ${\liminf}_{x \to \pm \infty} \sup c(x)$ for a real-valued function $c:\mathbb{R} \rightarrow \mathbb{R}$. 
\begin{itemize}
\item ${\limsup}_{x \to -\infty} {q_{p,s,1}}^{\prime}(x)=+\infty$  and  ${\liminf}_{x \to -\infty} {q_{p,s,1}}^{\prime}(x)=-\infty$. 

${\limsup}_{x \to +\infty} {q_{p,s,1}}^{\prime}(x)=0$  and  ${\liminf}_{x \to +\infty} {q_{p,s,1}}^{\prime}(x)=0$.
\item ${\limsup}_{x \to +\infty} {q_{p,s,2}}^{\prime}(x)=+\infty$  and  ${\liminf}_{x \to +\infty} {q_{p,s,2}}^{\prime}(x)=-\infty$.

 ${\limsup}_{x \to -\infty} {q_{p,s,2}}^{\prime}(x)=0$  and  ${\liminf}_{x \to -\infty} {q_{p,s,2}}^{\prime}(x)=0$.
\item ${\limsup}_{x \to \pm \infty} {q_{p,s,3}}^{\prime}(x)=+\infty$  and  ${\liminf}_{x \to \pm \infty} {q_{p,s,3}}^{\prime}(x)=-\infty$.
\end{itemize}
\end{itemize}
In this paper, we consider $p_{{\rm h},t}(x):=\sqrt{1+{(\tan  t)}^2 x^2}$ ($0<t<\frac{\pi}{2}$) and  $p_{{\rm h},t,q_{p,s,i}}(x):=\sqrt{1+\tan  t x^2}+q_{p,s,i}$, real-valued real analytic functions on $\mathbb{R}$. We put $0<t<\frac{\pi}{2}$ as a sufficiently large number such that $\frac{\pi}{2}$ is divisible by $2(\frac{\pi}{2}-t)$ in this paper. Note that the two asymptotes of the component of the hyperbola form an angle of size $2(\frac{\pi}{2}-t)$, around this. This is important in proving Theorem \ref{thm:1}. 
\subsection{Important smooth manifolds and maps.}
\label{subsec:3.4}

The canonical projection of the unit sphere $S^m$ is ${\pi}_{m+1,n} {\mid}_{S^m}$ and generalizes the natural height of $S^m$. Its image is the $n$-dimensional unit disk $D^n \subset {\mathbb{R}}^n$. We review a kind of construction generalizing the canonical projections of $S^m$, respecting the article \cite{kitazawa3} of the author and some preprints of the author related to this.

For $l>0$ ($n-1$)-dimensional smooth manifolds $S_j \subset {\mathbb{R}}^{n}$ with $n \geq 2$ which are also the zero sets of the smooth functions $f_j:{\mathbb{R}}^{n} \rightarrow \mathbb{R}$ and which are mutually disjoint, we also assume the relations $D_{\{S_j,f_j\}_{j=1}^l}:=\{x \in {\mathbb{R}}^n \mid f_j(x)>0, 1 \leq j \leq l\}$ and $\overline{D_{\{S_j,f_j\}_{j=1}^l}}^{{\mathbb{R}}^n}-D_{\{S_j,f_j\}_{j=1}^l}:={\sqcup}_{j=1}^l S_j$.
We can discuss as follows. Let $m \geq n$ be an integer. We define

$$X_{m,\{S_j,f_j\}_{j=1}^l}:=\{(x,y)=(x,y_1,\cdots y_{m-n+1}) \in {\mathbb{R}}^n \times {\mathbb{R}}^{m-n+1} \mid {\prod}_{j=1}^l (f_j(x))-{\Sigma}_{j=1}^{m-n+1} {y_j}^2=0 \}$$ and by implicit function theorem, it is an $m$-dimensional smooth submanifold of ${\mathbb{R}}^{m+1}$ with no boundary. More precisely, the value of the partial derivative of ${\prod}_{j=1}^l (f_j(x))-{\Sigma}_{j=1}^{m-n+1} {y_j}^2$ by some $y_j$ is not $0$, in the case $x \in D_{\{S_j,f_j\}_{j=1}^l}$. That of the function by some $x_b$ in $x=(x_1,\cdots x_n)$ is the product of the numbers $f_j(x) \neq 0$ except $j \neq a$ and the value of the 1st derivative of $f_a$ at $x$ in the case $x \in S_a$ for some $1 \leq a \leq l$, and the resulting value is not $0$. For this, see \cite{kitazawa4}, where the real algebraic or polynomial case is considered and this is a kind of pioneering studies on real algebraic construction of functions and maps. Consult also the preprint \cite{kitazawa6} of the author. We do not assume knowledge or arguments of \cite{kitazawa4, kitazawa6}. It is no problem.

In \cite{kitazawa8} (\cite[Theorem 1]{kitazawa8}) and the studies of the author \cite{kitazawa9, kitazawa10, kitazawa11, kitazawa12}, the case $(l,n)=(2,2)$ with $S_j$ being the graphs $\{(c_j(x),x) \mid x \in \mathbb{R}\}$ of the smooth or real analytic real-valued functions $c_j:\mathbb{R} \rightarrow \mathbb{R}$ with $c_1(x)<c_2(x)$ for $x \in \mathbb{R}$. The author has investigated topologies of Reeb spaces $R_{{\pi}_{m+1,1} {\mid}_{X_{m,\{S_j,f_j\}_{j=1}^l}}}$ exclusively for this specific case. In this specific case, $X_{m,\{S_j,f_j\}_{j=1}^l}$ is diffeomorphic to the product $S^{m-1} \times \mathbb{R}$ and we also consider a compactification of the real-valued function $c:={\pi}_{m+1,1} {\mid}_{X_{m,\{S_j,f_j\}_{j=1}^l}}:X:=X_{m,\{S_j,f_j\}_{j=1}^l} \rightarrow Y:=\mathbb{R}$ in the topology category. More precisely, the manifold of the domain of the resulting function $c_0$ is $X_0:=S^m$ with $\{(0,\cdots \pm 1)\} \subset {\mathbb{R}}^{m+1}$ being added to (the embedded image ${\phi}_X (X)$ in $X_0:=S^m$ of) $X_{m,\{S_j,f_j\}_{j=1}^l}$, the space of the target is $Y_0:=\mathbb{R}$, and ${\phi}_Y$ is the identity map, in Definition \ref{def:2}. As mentioned shortly around Theorem \ref{thm:1}, Theorem \ref{thm:1} is discussed partially in such situations. 

We consider the presented general case, in the present paper.

For example, Proposition \ref{prop:2} is of fundamental propositions. It is also essentially proven in the existing preprints of the author and can be also shown easily.
\begin{Prop}
\label{prop:2}
 A point of $X_{m,\{S_j,f_j\}_{j=1}^l}$ is a critical point of ${\pi}_{m+1,1} {\mid}_{X_{m,\{S_j,f_j\}_{j=1}^l}}$ if and only if it is a point mapped to a point $p$ of some $S_j$ whose tangent vector of $S_j$ at $p$ is always a tangent vector of $T_p (\{{\pi}_{n,1}(p)\} \times {\mathbb{R}}^{n-1})$ with $\{{\pi}_{n,1}(p)\} \times {\mathbb{R}}^{n-1} \subset {\mathbb{R}}^n$. Furthermore, to each of such points and each point of $S_j$, exactly one point of $X_{m,\{S_j,f_j\}_{j=1}^l}$ is mapped by ${\pi}_{m+1,n}$. 
\end{Prop}

We consider another case. For this, we respect arguments first presented in \cite{kitazawa6} and followed by \cite{kitazawa8}, and so on: we do not assume related non-trivial knowledge. Given $l>0$ $1$-dimensional real analytic manifolds ${S_j}^{\prime} \subset {\mathbb{R}}^{2}$ which are also the zero sets of the real analytic functions $g_{j}:{\mathbb{R}}^{2} \rightarrow \mathbb{R}$. 
\begin{enumerate}[label={\rm(}\text{\rm P3-}\arabic*{\rm )}]	
	\item \label{P3-1} (Defining the region $D_{\{{S_j}^{\prime},g_j\}_{j=1}^l} \subset {\mathbb{R}}^2$.)
	 $D_{\{{S_j}^{\prime},g_j\}_{j=1}^l}:=\{x \in {\mathbb{R}}^2 \mid g_j(x)>0, 1 \leq j \leq l\}$.
	\item \label{P3-2} (Each ${S_j}^{\prime}$ is regarded as the intersection of ${\overline{D_{\{{S_j}^{\prime},g_j\}_{j=1}^l}}}^{{\mathbb{R}}^2}$ and the zero set of $g_j$.) ${S_j}^{\prime}={\overline{D_{\{{S_j}^{\prime},g_j\}_{j=1}^l}}}^{{\mathbb{R}}^2} \bigcap \{x \mid g_j(x)=0\}$.
	\item \label{P3-3} (Transversality of ${S_{j_1}}^{\prime} \ni p$ and ${S_{j_2}}^{\prime} \ni p$ at $p \in {\overline{D_{\{{S_j}^{\prime},g_j\}_{j=1}^l}}}^{{\mathbb{R}}^2}-D_{\{{S_j}^{\prime},g_j\}_{j=1}^l}$.)
	Each point of ${\overline{D_{\{{S_j}^{\prime},g_j\}_{j=1}^l}}}^{{\mathbb{R}}^2}-D_{\{{S_j}^{\prime},g_j\}_{j=1}^l}$ is contained in at most two manifolds ${S_j}^{\prime}$. Furthermore, for a point $p$ there contained in such two distinct manifolds ${S_{j_1}}^{\prime} \ni p$ and ${S_{j_2}}^{\prime} \ni p$, the non-zero tangent vector of ${S_{j_1}}^{\prime}$ there and the non-zero tangent vector of ${S_{j_2}}^{\prime}$ there are independent. This shows so-called transversality of intersection of  ${S_{j_1}}^{\prime} \ni p$ and ${S_{j_2}}^{\prime} \ni p$ at $p$.
	\item \label{P3-4} (For data for reconstructing a real analytic map onto ${\overline{D_{\{{S_j}^{\prime},g_j\}_{j=1}^l}}}^{{\mathbb{R}}^2}$.)
	 For each manifold ${S_j}^{\prime} \subset {\mathbb{R}}^{2}$ such that ${\overline{D_{\{{S_j}^{\prime},g_j\}_{j=1}^l}}}^{{\mathbb{R}}^2} \bigcap {S_{j}}^{\prime}$ is non-empty, either $1$ or $2$ is assigned and for the two manifolds ${S_{j_1}}^{\prime}$ and ${S_{j_2}}^{\prime}$ intersecting as above, distinct numbers $I_{{S_{j_1}}^{\prime}} \in \{1,2\}$ and  $I_{{S_{j_2}}^{\prime}} \in \{1,2\}$ are assigned.  For each manifold ${S_j}^{\prime} \subset {\mathbb{R}}^{2}$ such that ${\overline{D_{\{{S_j}^{\prime},g_j\}_{j=1}^l}}}^{{\mathbb{R}}^2} \bigcap {S_{j}}^{\prime}$ is empty,  $I_{{S_{j}}^{\prime}}=0$ is assigned.
\end{enumerate}

Let $m_1, m_2 \geq 1$ be integers.

We define \\

$X_{m_1,m_2,\{{S_j}^{\prime},g_j\}_{j=1}^l,I}:\\=\{(x,y_1,y_2)=(x,(y_{1,1} \ldots y_{1,m_1}),(y_{2,1} \ldots y_{2,m_2})) \\ \in {\mathbb{R}}^2 \times {\mathbb{R}}^{m_1} \times {\mathbb{R}}^{m_2} \mid {\prod}_{j \in \{I_{{S_j}^{\prime}}=i\}} (g_j(x))-{\Sigma}_{j=1}^{m_i} {y_{i,j}}^2=0\ \text{for}\ i=1,2\}$\\ and by implicit function theorem, we have Proposition \ref{prop:3}.
\begin{Prop}
\label{prop:3}
$X_{m_1,m_2,\{{S_j}^{\prime},g_j\}_{j=1}^l,I}$ is an {\rm (}$m_1+m_2${\rm )}-dimensional real analytic manifold in ${\mathbb{R}}^{m_1+m_2+2}$.
\end{Prop}
\begin{proof}

More precisely, we investigate the differential of the map 
$e_{\{{S_j}^{\prime},g_j\}}:=(e_{\{{S_j}^{\prime},g_j\},1},e_{\{{S_j}^{\prime},g_j\},2}):{\mathbb{R}}^{m_1+m_2+2} \rightarrow {\mathbb{R}}^2$ defined by
$e_{\{{S_j}^{\prime},g_j\}}(x,y_1,y_2):=(e_{\{{S_j}^{\prime},g_j\},1}(x,y_1,y_2),e_{\{{S_j}^{\prime},g_j\},2}(x,y_1,y_2)):=({\prod}_{j \in \{I_{{S_j}^{\prime}}=1\}} (g_j(x))-{\Sigma}_{j=1}^{m_1} {y_{1,j}}^2,{\prod}_{j \in \{I_{{S_j}^{\prime}}=2\}} (g_j(x))-{\Sigma}_{j=1}^{m_2} {y_{2,j}}^2)$. We prove that at each point $(x_{0},y_{1,0},y_{2,0}) \in X_{m_1,m_2,\{{S_j}^{\prime},g_j\}_{j=1}^l}$ the differential is of rank $2$.\\
\ \\
Case P3-1 $x_{0} \in D_{\{{S_j}^{\prime},g_j\}_{j=1}^l}$. \\
The values of partial derivatives $\frac{\partial e_{\{{S_j}^{\prime},g_j\},1}}{\partial y_{2,j}}$ and $\frac{\partial e_{\{{S_j}^{\prime},g_j\},2}}{\partial y_{1,j}}$ are always $0$. The values of some partial derivative of the form $\frac{\partial e_{\{{S_j}^{\prime},g_j\},1}}{\partial y_{1,j}}$ and one of the form $\frac{\partial e_{\{{S_j}^{\prime},g_j\},2}}{\partial y_{2,j}}$ are non-zero. We can see that at the point the rank of the differential is $2$.\\

\noindent Case P3-2 $x_{0}$ is in exactly one ${S_{i}}^{\prime}$ with $I_{{S_i}^{\prime}}=i_0$. \\
The value of a partial derivative of the form $\frac{\partial e_{\{{S_j}^{\prime},g_j\},i_0}}{\partial y_{i,j}}$ ($i=1,2$) is always $0$. The value of some partial derivative of the form $\frac{\partial e_{\{{S_j}^{\prime},g_j\},{i_0}^{\prime}}}{\partial y_{{i_0}^{\prime},j}}$ (${i_0}^{\prime} \neq i_0$) is non-zero. The value of some partial derivative of the form $\frac{\partial e_{\{{S_j}^{\prime},g_j\},i_0}}{\partial x_j}$ is non-zero, by \ref{P3-1} and \ref{P3-2}. We can see that at the point the rank of the differential is $2$.\\

\noindent Case P3-3 $x_{0}$ is in two distinct ${S_{i_1}}^{\prime}$ and ${S_{i_2}}^{\prime}$. \\
The value of a partial derivative $\frac{\partial e_{\{{S_j}^{\prime},g_j\},i_a}}{\partial y_{i_b,j}}$ is always $0$. The value of some partial derivative of the form $\frac{\partial e_{\{{S_j}^{\prime},g_j\},1}}{\partial x_i}$ is non-zero. The value of some partial derivative of the form $\frac{\partial e_{\{{S_j}^{\prime},g_j\},2}}{\partial x_i}$ is non-zero. 
At the point, as assumed in \ref{P3-3}, ${S_{i_1}}^{\prime}$ and ${S_{i_2}}^{\prime}$ intersect satisfying the transversality. Remember also \ref{P3-4}. We can see that at the point the rank of the differential is $2$. \\
\ \\
From Cases P3-1, P3-2 and P3-3, the rank of the differential of $e_{\{{S_j}^{\prime},g_j\}}$ is $2$ at each point $(x_{0},y_{1,0},y_{2,0}) \in X_{m_1,m_2,\{{S_j}^{\prime},g_j\}_{j=1}^l}$. This completes the proof. 
\end{proof}
Proposition \ref{prop:4} can be shown immediately, by observing (local) structures of the functions and maps.
\begin{Prop}
	\label{prop:4}
	A point of $X_{m_1, m_2,\{{S_j}^{\prime},g_j\}_{j=1}^l,I}$ is a critical point of ${\pi}_{m_1+m_2+2,1} {\mid}_{X_{m_1,m_2,\{{S_j}^{\prime},g_j\}_{j=1}^l,I}}$ if and only if it is a point satisfying at least one of the following hold.
		\begin{itemize}
	\item It is mapped to a point $p$ of some ${S_j}^{\prime}$ whose tangent vector of ${S_j}^{\prime}$ at $p$ is always a tangent vector of $T_p (\{{\pi}_{2,1}(p)\} \times {\mathbb{R}})$ with $\{{\pi}_{2,1}(p)\} \times \mathbb{R} \subset {\mathbb{R}}^2$.
		\item It is mapped to a point $p$ contained in exactly two manifolds ${S_{j}}^{\prime}$. We can also see that to each of such points, exactly one point of $X_{m_1,m_2,\{{S_j}^{\prime},fg_j\}_{j=1}^l,I}$ is mapped by ${\pi}_{m_1+m_2+2,2}$. 
	\end{itemize}		
				 
\end{Prop}

\section{On our main result.}

\subsection{Our proof of Theorem \ref{thm:1}.}
We prove Theorem \ref{thm:1} and present its additional remark, Theorem \ref{thm:3}.
\begin{proof}[A proof of Theorem \ref{thm:1}]
In the Subsection \ref{subsec:3.3}, we also have a component of hyperbola $\{(x,p_{{\rm h},t}(x)) \mid x \in \mathbb{R}\} \subset {\mathbb{R}}^2$ and by rotating this around the origin $0$ by degree $(\frac{\pi}{2}-t) \times 2j$ for $0 \leq 2j< \frac{2\pi}{\frac{\pi}{2}-t}$ with $2j$ being even integers.
We obtain an explicit case of\\ $X_{m,\{S_j,f_j\}_{j=1}^l}:=\{(x,y)=(x,y_1,\cdots y_{m-1}) \in {\mathbb{R}}^n \times {\mathbb{R}}^{m-n+1} \mid {\prod}_{j=1}^l (f_j(x))-{\Sigma}_{j=1}^{m-n+1} {y_j}^2=0 \}$, defined in the Subsection \ref{subsec:3.4}. Put $l=\frac{\pi}{(\frac{\pi}{2}-t)}$
We define the $j$-th manifold $S_j$ to be the ($j+1$)-th component of a hyperbola here except fintiely many $S_j$ and we define the remaining $S_j$ in the following way. For the others, we do not change the angles of the rotations and we only change the initial $1$-dimensional real algebraic manifold slightly.

\begin{itemize}
\item Exactly $n_{\rm d}$ manifolds $S_j$ undefined here satisfies $\frac{3\pi}{2(\frac{\pi}{2}-t)}<2j<\frac{2\pi}{\frac{\pi}{2}-t}$ defined by considering $\{(x,p_{{\rm h},t,q_{p,s,1}}(x)) \mid x \in \mathbb{R}\} \subset {\mathbb{R}}^2$ before the rotation, instead of $p_{{\rm h},t}$.
\item Exactly $n_{\rm e}$ manifolds $S_j$ undefined here satisfies $0<2j<\frac{\pi}{2(\frac{\pi}{2}-t)}$ defined by considering $\{(x,p_{{\rm h},t,q_{p,s,1}}(x)) \mid x \in \mathbb{R}\} \subset {\mathbb{R}}^2$ before the rotation, instead of $p_{{\rm h},t}$.

\item  Exactly $n_{\rm c}$ manifolds $S_j$ undefined here satisfies $\frac{\pi}{2(\frac{\pi}{2}-t)}<2j<\frac{\pi}{\frac{\pi}{2}-t}$ defined by considering $\{(x,p_{{\rm h},t,q_{p,s,1}}(x)) \mid x \in \mathbb{R}\} \subset {\mathbb{R}}^2$ before the roation, instead of $p_{{\rm h},t}$.
\end{itemize} 
In addition, we define $f_1(x_1,x_2)=p_{{\rm h},t}(x_1)-x_2, p_{{\rm h},t,q_{p,s,i}}(x_1)-x_2$ first, and the remaining $f_j$ are defined canonically, according to the rotations.

Note that the number and symmetry of the arrangement of the manifolds $S_j \subset {\mathbb{R}}^2$ in the proof is not essential in our proof. Simply, we do this respecting symmetry. We also define $q_{p,s,i}$ ($i=2,3$) in the Subsection \ref{subsec:3.3}, where we do not use this in this paper.

We can obtain $X_{m,\{S_j,f_j\}_{j=1}^l}$ suitably and have an open disk $D_{R,{\rm o}}:=\{(x_1,x_2) \mid {x_1}^2+{x_2}^2<R>0\}$ in such a way that $X_{m,\{S_j,f_j\}_{j=1}^l}-{{{\pi}_{m+1,2}} {\mid}_{X_{m,\{S_j,f_j\}_{j=1}^l}}}^{-1}(D_{R,{\rm o}})$ consists of $l$ $m$-dimensional connected smooth manifolds. We discuss compactifications of the restrictions of ${\pi}_{m+1,1} {\mid}_{X_{m,\{S_j,f_j\}_{j=1}^l}}$ to these $m$-dimensional manifolds and a compactification of  ${\pi}_{m+1,1} {\mid}_{X_{m,\{S_j,f_j\}_{j=1}^l}}$. 

First, we forget exactly $n_{\rm c}$ connected manifolds containing some points mapped to points of manifolds $S_j$ with $\frac{\pi}{2(\frac{\pi}{2}-t)}<2j<\frac{\pi}{\frac{\pi}{2}-t}$ by ${\pi}_{m+1,2}$ and containing infinitely and countably many points of the form $\{(x,p_{{\rm h},t,q_{p,s,i}}(x),0) \in {\mathbb{R}}^2 \times {\mathbb{R}}^{m-1} \mid x \in S(p_{{\rm h},t,q_{p,s,i}}(x))\}$ suitably rotated, and exactly $n_{\rm c}$ connected manifolds containing some points mapped to points of manifolds $S_j$ with $\frac{\pi}{(\frac{\pi}{2}-t)}<2j<\frac{3\pi}{2(\frac{\pi}{2}-t)}$. For the remaining $l-2n_{\rm c}$ connected manifolds, we can consider the one-point compactifications to have manifolds diffeomorphic to $D^m$, and compactifications of the restrictions of ${\pi}_{m+1,1} {\mid}_{X_{m,\{S_j,f_j\}_{j=1}^l}}$ there in the topology category. 

We consider the $n_{\rm c}$ pairs of the $m$-dimensional connected manifolds each of which consists of one containing some points mapped to some points in $S_j$ with $\frac{\pi}{2(\frac{\pi}{2}-t)}<2j<\frac{\pi}{\frac{\pi}{2}-t}$ by ${\pi}_{m+2,2}$ and one containing some points mapped to some points in $S_j$ with $\frac{\pi}{(\frac{\pi}{2}-t)}<2j<\frac{3\pi}{2(\frac{\pi}{2}-t)}$ by ${\pi}_{m+2,2}$. From them, we can have $n_{\rm c}$ compactifications by adding circles to have connected manifolds homeomorphic to $S^{m-1} \times D^1$. Compactifications of the restrictions of ${\pi}_{m+1,1} {\mid}_{X_{m,\{S_j,f_j\}_{j=1}^l}}$ there in the topology category are also naturally obtained. 

In short, by $X_{m,n_{\rm d, e, c}}:=X_{m,\{S_j,f_j\}_{j=1}^l}$, we have a desired case.
More precisely, we explain a Reeb-D-C of ${\pi}_{m+1,1} {\mid}_{X_{m,n_{\rm d, e, c}}}$ by using notation in Definition \ref{def:4}. We have $X:=X_{m,n_{\rm d, e, c}}$ and $Y:=\mathbb{R}$. We need the embedding ${\phi}_X:X_{m,n_{\rm d, e, c}}:=X_{m,\{S_j,f_j\}_{j=1}^l} \rightarrow X_0$ as a canonical smooth embedding into a smooth manifold $X_0$ diffeomorphic to $S^m$ in the case $n_{\rm c}=0$ or a connected sum of $n_{\rm c}>0$ copies of $S^1 \times S^{m-1}$ taken in the smooth category. We also need a natural smooth embedding ${\phi}_Y:\mathbb{R} \rightarrow S^1=\{(\cos \theta, \sin \theta) \mid 0 \leq \theta<2\pi\}$ satisfying ${\phi}_Y(\mathbb{R})=S^1-\{(0,-1)\}$. 

We explain the structure of a $1$-dimensional CW complex for a Reeb-D-C of ${\pi}_{m+1,1} {\mid}_{X_{m,n_{\rm d, e, c}}}$. Remember arguments on critical points of $p_{{\rm h},t}$ and $p_{{\rm h},t,q_{p,s,i}}$. Refer to \cite[Theorem 3.1]{saeki1}, \cite[Theorems 2.1 and 2.8]{saeki2} and \cite[Theorem 7.5]{gelbukh1}. We explain these papers shortly. Conditions for Reeb spaces to have the structures of graphs, or (equivalently,) 1-dimensional complexes which are (locally) finite, are discussed in \cite{saeki1, saeki2}. For manifolds or more generally, connected, compact, separable and locally connected spaces ({\it Peano continua}), Reeb spaces of continuous functions there are also of such spaces (\cite{gelbukh1}).

This completes the proof.

\end{proof}

\begin{Thm}
\label{thm:3}
Theorem \ref{thm:1} is realized by the following.
\begin{enumerate}
\item $X_{m,n_{\rm d,e,c}} \subset {\mathbb{R}}^{m+1}$ can be chosen by a suitable $X_{m,\{S_j,f_j\}_{j=1}^l} \subset {\mathbb{R}}^{m+1}$ with $S_j \subset {\mathbb{R}}^2$ and a suitable real analytic function $f_j:{\mathbb{R}}^2 \rightarrow \mathbb{R}$.
\item We have a compactification of the smooth function ${\pi}_{m+1,1} {\mid}_{X_{m,n_{\rm d,e,c}}}:X:=X_{m,n_{\rm d,e,c}} \rightarrow Y:=\mathbb{R}$ in the topology category as a continuous map $c_0:X_0 \rightarrow Y_0$: the manifold $X_0$ is homeomorphic to $S^m$ in the case $n_{\rm c}=0$ or a connected sum of $n_{\rm c}>0$ copies of $S^1 \times S^{m-1}$ taken in the smooth category, $Y_0:=S^1$ and for the embedding ${\phi}_{Y}:Y:=\mathbb{R} \rightarrow Y_0$ in Definition \ref{def:4} satisfies $Y_0-{\phi}_Y(\mathbb{R})=\{(0,-1)\} \subset S^1$. We also have the Reeb-D-C of ${\pi}_{m+1,1}  {\mid}_{X_{m,n_{\rm d,e,c}}}$ canonically.
\end{enumerate}
\end{Thm}

\subsection{Our proof of Theorem \ref{thm:2}.}
We prove Theorem \ref{thm:2} and present its additional remark, Theorem \ref{thm:4}.
\begin{proof}[A proof of Theorem \ref{thm:2}]

STEP 2-1 Defining a case for $X_{m,1}$. \\
We first set a case for $X_{m,1}$.
We also choose two suitable integers $m_1>0$ and $m_2>0$ with $m=m_1+m_2$.
We define three $1$-dimensional real algebraic manifolds ${S_{1,1}}^{\prime}, {S_{1,2}}^{\prime}, {S_{1,3}}^{\prime} \subset {\mathbb{R}}^2$ and three real analytic functions $g_{1,1}(x_1,x_2)$, $g_{1,2}(x_1,x_2)$ and $g_{1,3}(x_1,x_2)$ as follows.
\begin{enumerate}[label={\rm(}\text{\rm T2-}\arabic*{\rm )}]
\item \label{T2-1}
${S_{1,1}}^{\prime}=\{(e^{-x^2}{\sin}^2 x,x) \mid x \in \mathbb{R}\}$. $g_{1,1}(x_1,x_2):=x_1-e^{-{x_2}^2}{\sin}^2 x_2$.
\item \label{T2-2} ${S_{1,2}}^{\prime}=\{(\frac{a_2}{{(x+a_1)}^2+1},x) \mid x \in \mathbb{R}\}$. $g_{1,2}(x_1,x_2):=\frac{a_2}{{(x_2+a_1)}^2+1}-x_1$ for suitably chosen positive numbers $a_1>0$ and $a_2>0$. 
\item \label{T2-3} 
${S_{1,3}}^{\prime}=\{(x,a_3) \mid x \in \mathbb{R}\}$. $g_{1,3}(x_1,x_2):=a_3-x_2$ for a suitably chosen number $a_3>-a_1$ with $a_3<0$. 
\item \label{T2-4}
We can choose ${S_{j}}^{\prime}:={S_{1,j}}^{\prime}$ and $g_{j}:={g_{1,j}}$ to have a case of $D_{\{{S_j}^{\prime},g_j\}_{j=1}^3} \subset {\mathbb{R}}^2$ before. We can have this in such a way that ${S_{1,j}}^{\prime} \bigcap {\overline{D_{\{{S_j}^{\prime},g_j\}_{j=1}^l}}}^{{\mathbb{R}}^2}$ is non-empty if and only if $j=1,2$. We can also have this in such a way that ${S_{1,1}}^{\prime} \bigcap {S_{1,2}}^{\prime} \bigcap {\overline{D_{\{{S_j}^{\prime},g_j\}_{j=1}^l}}}^{{\mathbb{R}}^2}$ is a one-point set of the form $\{(\frac{a_2}{{(a_4+a_1)}^2+1},a_4)=(e^{-{a_4}^2}{\sin}^2 a_4,a_4)\}$ with $-a_1<a_4<a_3<0$. 
\item \label{T2-5}
We can also assign $I_{S_{1,j}^{\prime}}=j$ for $j=1,2$ and $I_{S_{1,3}^{\prime}}=0$, to have $X_{m_1,m_2,\{{S_j}^{\prime},g_j\}_{j=1}^3,I}$ and we define this as $X_{m,1}$.
\end{enumerate}
STEP 2-2 STEP 2-1 Defining a case for $X_{m,2}$ where some $g_2$ is not real analytic on a subset of Lebesgue measure $0$. \\
Next we set a case for $X_{m,2}$.

We define two $1$-dimensional real algebraic manifolds ${S_{2,1}}^{\prime}, {S_{2,2}}^{\prime} \subset {\mathbb{R}}^2$ and two polynomials $g_{2,1}(x_1,x_2)$ and $g_{2,2}(x_1,x_2)$ as follows.
\begin{enumerate}[start=6, label={\rm(}\text{\rm T2-}\arabic*{\rm )}]
\item \label{T2-6} ${S_{2,1}}^{\prime}=\{(e^{-\frac{1}{x^2}}{\sin}^2 \frac{1}{x},x)  \mid x \in \mathbb{R}-\{0\}\} \sqcup \{(0,0)\}$. $g_{2,1}(x_1,x_2):=x_1-e^{-\frac{1}{{x_2}^2}}{\sin}^2 \frac{1}{x_2}$ for $x_2 \neq 0$ and $g_{2,1}(x_1,x_2)=x_1$ for $x_2=0$.
\item \label{T2-7} $\{(x_1,x_2) \mid g_{2,2}(x_1,x_2) \geq 0\} \subset {\mathbb{R}}^2$ is an ellipsoid and ${S_{2,2}}^{\prime}:=\{(x_1,x_2) \mid g_{2,2}(x_1,x_2)=0\}$.
\item \label{T2-8}  ${S_{2,2}}^{\prime} \bigcap {\overline{D_{\{{S_j}^{\prime},g_j\}_{j=1}^2}}}^{{\mathbb{R}}^2}$ is the graph of the form $\{(g(x),x) \mid 0 \leq x\leq -\frac{1}{a_4}\}$ with a suitable smooth function $g:\{x \mid -\epsilon<x<-\frac{1}{a_4}+\epsilon\} \rightarrow \mathbb{R}$ satisfying the following.

\begin{enumerate}[label={\rm(}\text{\rm T2-8-}\arabic*{\rm )}]
\item \label{T2-8-1} $\epsilon>0$ is a sufficiently small positive number.
\item \label{T2-8-2} $g$ has exactly one critical point.
\item \label{T2-8-3} There exists a unique point where $g$ has a maximum. The unique point is $x:=x_0:=\frac{1}{a_1}$. 
\end{enumerate}
\item \label{T2-9} We can choose ${S_{j}}^{\prime}:={S_{2,j}}^{\prime}$ and $g_{j}:={g_{2,j}}$ to have a case of $D_{\{{S_j}^{\prime},g_j\}_{j=1}^2} \subset {\mathbb{R}}^2$ before. We have this in such a way that ${S_{2,j}}^{\prime} \bigcap {\overline{D_{\{{S_j}^{\prime},g_j\}_{j=1}^2}}}^{{\mathbb{R}}^2}$ is non-empty for $j=1,2$.
\item \label{T2-10}
We can also assign $I_{S_{2,j}^{\prime}}=j$ for $j=1,2$, to have $X_{m_1,m_2,\{{S_j}^{\prime},g_j\}_{j=1}^2,I}$. We define $X_{m,2}:=X_{m_1,m_2,\{{S_j}^{\prime},g_j\}_{j=1}^2,I}$ with $m=m_1+m_2$. 
\end{enumerate}
\noindent STEP 2-3 The manifold $X_{m,1}$, the function ${\pi}_{m+2,1} {\mid}_{X_{m,1}}$ and the Reeb space $R_{{\pi}_{m+2,1} {\mid}_{X_{m,1}}}$. \\
Hereafter, we may refer to related preprints of the author to some extent, mainly \cite{kitazawa8}, and \cite{kitazawa9, kitazawa10, kitazawa11, kitazawa12}. We also discuss important arguments in a self-contained way. We consider the restriction of the projection to the second component, denoted by ${\pi}_{m+2,1,2}:{\mathbb{R}}^{m+2} \rightarrow \mathbb{R}$. We investigate ${\pi}_{m+2,1,2} {\mid}_{X_{m,1}}$ to investigate $X_{m,1}$. We have a proper smooth function whose image is diffeomorphic to $\{x \mid x>0\}$ and which is with exactly one critical point and with regular level sets being empty or diffeomorphic to $S^{m-1}$. 
From Proposition \ref{prop:4}, by a fundamental argument on differential topology, $X_{m,1}$ is diffeomorphic to ${\mathbb{R}}^m$.  We investigate ${\pi}_{m+2,1} {\mid}_{X_{m,1}}$. Its image is $\{x \mid 0 \leq x \leq a_2\}$. On the preimage ${{\pi}_{m+2,1} {\mid}_{X_{m,1}}}^{-1}(\{x \mid x>0\})$, it is proper and by Proposition \ref{prop:4}, the image of the critical set of the restriction there is closed and discrete in $\{x \mid x>0\}$. In addition, the Reeb digraph $R_{{{\pi}_{m+2,1} {\mid}_{X_{m,1}}}^{-1}(\{x \mid x>0\})}$ can be defined as a $1$-dimensional locally finite CW complex the closure of each of whose edges is homeomorphic to $D^1$. The preimage ${{\pi}_{m+2,1} {\mid}_{X_{m,1}}}^{-1}(\{0\})$ is a discrete subset of $X_{m,1}$ and mapped onto a discrete subset of $R_{{\pi}_{m+2,1} {\mid}_{X_{m,1}}}$ by the restriction of the quotient map $q_{{\pi}_{m+2,1} {\mid}_{X_{m,1}}}$, in an injective way. For each point in the discrete subset of $R_{{\pi}_{m+2,1} {\mid}_{X_{m,1}}}$, due to the structure of the manifolds and maps, we have a small open connected neighborhood homeomorphic to $\{x \mid 0 \leq x <1\}$ containing no vertex from the Reeb (di)graph of ${\pi}_{m+2,1} {\mid}_{{{\pi}_{m+1,1} {\mid}_{X_{m,1}}}^{-1}(\{x \mid x>0\})}$. In addition, we can see that arbitrary distinct two points of $R_{{\pi}_{m+2,1} {\mid}_{X_{m,1}}}$ can be separated by (small and connected) open neighborhoods in $R_{{\pi}_{m+2,1} {\mid}_{X_{m,1}}}$.

From this, the Reeb space is Hausdorff and regarded to be the Reeb digraph of ${\pi}_{m+2,1} {\mid}_{X_{m,1}}$, which is also locally finite and the closure of each of whose edges is homeomorphic to $D^1$.\\
\ \\
\noindent STEP 2-4 The manifold $X_{m,2}$, the function ${\pi}_{m+2,1} {\mid}_{X_{m,2}}$ and the Reeb space $R_{{\pi}_{m+2,1} {\mid}_{X_{m,2}}}$. \\

We investigate ${\pi}_{m+2,1,2} {\mid}_{X_{m,2}}$ to investigate $X_{m,2}$. 
From Proposition \ref{prop:4}, we have a proper smooth function whose image is diffeomorphic to $D^1$ and which is with exactly two critical points and with regular level sets being empty or diffeomorphic to $S^{m-1}$. From this, by a fundamental argument on differential topology, $X_{m,2}$ is diffeomorphic to ${\mathbb{R}}^m$.  We investigate ${\pi}_{m+2,1} {\mid}_{X_{m,2}}$. Its image is $\{x \mid 0 \leq x \leq a_2\}$. On the preimage ${{\pi}_{m+2,1} {\mid}_{X_{m,2}}}^{-1}(\{x \mid x>0\})$, it is proper and by Proposition \ref{prop:4}, the image of the critical set of the restriction is closed and discrete in $\{x \mid x>0\}$. In addition, the Reeb digraph $R_{{{\pi}_{m+2,1} {\mid}_{X_{m,2}}}^{-1}(\{x \mid x>0\})}$ can be defined as a $1$-dimensional locally finite CW complex the closure of each of whose edges is homeomorphic to $D^1$. The preimage ${{\pi}_{m+2,1} {\mid}_{X_{m,2}}}^{-1}(\{0\})$ is a closed subset of $X_{m,2}$ which contains countably many points and the origin $0 \in {\mathbb{R}}^{m+2}$, and mapped onto a subset in $R_{{\pi}_{m+2,1} {\mid}_{X_{m,2}}}$ by the restriction of the quotient map $q_{{\pi}_{m+2,1} {\mid}_{X_{m,2}}}$ there, in an injective way. Due to the structure of the manifolds and maps, for each of the point of the subset except the point $q_{{\pi}_{m+2,1} {\mid}_{X_{m,2}}}(0)$, we have some small open connected neighborhood of $R_{{\pi}_{m+2,1} {\mid}_{X_{m,2}}}$ containing no vertex from the Reeb (di)graph of ${\pi}_{m+2,1} {\mid}_{{{\pi}_{m+2,1} {\mid}_{X_{m,2}}}^{-1}(\{x \mid x>0\})}$ and at most one point of $q_{{\pi}_{m+2,1} {\mid}_{X_{m,2}}}({{\pi}_{m+2,1} {\mid}_{X_{m,2}}}^{-1}(\{0\}))$. In addition, every open neighborhood of $q_{{\pi}_{m+2,1} {\mid}_{X_{m,2}}}(0)$ in $R_{{\pi}_{m+2,1} {\mid}_{X_{m,2}}}$ must contain infinitely and countably many points from $q_{{\pi}_{m+2,1} {\mid}_{X_{m,2}}}({{\pi}_{m+2,1} {\mid}_{X_{m,2}}}^{-1}(\{0\})) \subset R_{{\pi}_{m+2,1} {\mid}_{X_{m,2}}}$.   

From this, with \cite[Theorem 7.5]{gelbukh1}, the Reeb space is a Peano continuum and regarded to be the Reeb digraph of ${\pi}_{m+2,1} {\mid}_{X_{m,2}}$, with exactly one non-finite point $q_{{\pi}_{m+2,1} {\mid}_{X_{m,2}}}(0)$. \\
\ \\
STEP 2-5 Discussing the properties (\ref{thm:2.3}) and (\ref{thm:2.4}): in the previous arguments the properties (\ref{thm:2.1}) and (\ref{thm:2.2}) are already shown.\\

In short, from Proposition \ref{prop:4}, the coordinates, and our construction, Properties (\ref{thm:2.3}) and (\ref{thm:2.4}) are also immediately shown. \\
\ \\
This completes the proof.

\end{proof}
\begin{Thm}
\label{thm:4}
Theorem \ref{thm:2} is realized by the following.
\begin{enumerate}
\item $X_{m,1}$ can be chosen to be a real analytic manifold of the form $X_{m_1,m_2,\{{S_j}^{\prime},g_j\}_{j=1}^3,I}$ for positive integers $m_1$ and $m_2$ with $m=m_1+m_2$.
\item $X_{m,2}$ can be chosen to be a smooth manifold of the form $X_{m_1,m_2,\{{S_j}^{\prime},g_j\}_{j=1}^2,I}$ for a smooth function $g_1:{\mathbb{R}}^2 \rightarrow \mathbb{R}$ being real-analytic and elementary outside a subset of Lebesgue measure $0$ in ${\mathbb{R}}^2$, a real analytic and elementary function $g_2:{\mathbb{R}}^2 \rightarrow \mathbb{R}$, and positive integers $m_1$ and $m_2$ with $m=m_1+m_2$.
\end{enumerate}
\end{Thm}
\subsection{Additional exposition.}
We present a future problem as follows.
\begin{Prob}
	\label{prob:3}
	Can we have any $1$-dimensional connected and finite CW complexes whose edges are oriented as the graph diagram for NF of the Reeb graph of a smooth or real analytic functions similar to ones in the present paper?

\end{Prob}
This is especially related to Theorem \ref{thm:1} (\ref{thm:3}). It has been shown in the case where the $1$-dimensional CW complex is connected and has exactly one vertex.	

Remark \ref{rem:1} is related to Theorem \ref{thm:2} (\ref{thm:4}). 
\begin{Rem}
	\label{rem:1}
$X_{m,2}$ and ${\pi}_{m+1,1} {\mid}_{X_{m,2}}$, presented in Theorem \ref{thm:2}	(\ref{thm:4}), are similar to ones in \cite[Main Theorem 2]{kitazawa8}. \cite[Main Theorem 2]{kitazawa8} is a pioneering example of smooth submanifolds with no boundary in Euclidean spaces represented as the zero sets of some smooth maps being real analytic and elementary functions outside some subsets of Lebesgue measure $0$ and has been constructed by the author himself.

Situations are also a bit different. For example, ${S_{2,2}}^{\prime} \bigcap {\overline{D_{\{{S_j}^{\prime},g_j\}_{j=1}^2}}}^{{\mathbb{R}}^2}$ is not of the form of a graph of the form $\{(g(x),x) \mid 0 \leq x\leq -\frac{1}{a_4}\}$, in \cite[Main Theorem 2]{kitazawa8}.

Note that roughly, Theorems \ref{thm:2} may be regarded to be a theorem on relationship between \cite[Main Theorem 1]{kitazawa8} and \cite[Main Theorem 2]{kitazawa8}.
\end{Rem}
\section{Conflict of interest and Data availability.}
 The author is a researcher at Osaka Central Advanced Mathematical Institute (OCAMI researcher). This is supported by MEXT Promotion of Distinctive Joint Research Center Program JPMXP0723833165. He thanks the people for the hospitality, where he is not employed by the institute or the projects. 
  %Some of works by other researchers and this version may overlap in some of the contents due to the nature that our problems are natural in theory of Morse functions and applications to differential topology and that related mathematical studies are very fundamental and classical in some senses, for example. However the present version of our paper is presented independent of these work. \\
  %Saga Souhatsu Mathematical Seminar (http://inasa.ms.saga-u.ac.jp/Japanese/saga-souhatsu.html), inviting the author as a speaker, is funded and supported by JST Fusion Oriented REsearch for disruptive Science and Technology JPMJFR202U: the author was a speaker on 2024/7/12 supported by this project.\\
 
No data other than the present file is generated, related to the present study. We do not assume non-trivial arguments in preprints which are still unpublished. We may refer to these preprints to some extent.

\end{document}